\documentclass[12pt,hidelinks,oneside]{amsart}
\usepackage{enumitem}
\usepackage{hyperref}
\usepackage{geometry}
\usepackage{amssymb}
\usepackage{parskip}
\usepackage{pdfpages}
\newtheorem{thm}[equation]{\textbf{Theorem}}
\newtheorem{lem}[equation]{\textbf{Lemma}}

\newtheorem{pro}[equation]{\textbf{Proposition}}
\newtheorem{rem}[equation]{\textbf{Remark}}
\newtheorem{cor}[equation]{\textbf{Corollary}}
\newtheorem{defn}[equation]{\textbf{Definition}}
\newtheorem{hyp}[equation]{\textbf{Hypothesis}}
\numberwithin{equation}{section}

\title{Bounded Weak Solutions of Degenerate $p$-Poisson Equations}
\author{Sullivan F. MacDonald and Scott Rodney }
\date{July 2023}
\thanks{S. Rodney was supported by the NSERC Discovery Grant Program. S.F. MacDonald was supported by the NSERC USRA Program.}

\address{Scott Rodney\\
Dept. of Mathematics, Physics and Geology \\ 
Cape Breton University \\
Sydney, NS B1Y3V3, CA} 
\email{scott\_rodney@cbu.ca}
\address{Sullivan F. MacDonald, MSc student in the Dept. of Mathematics \& Statistics\\ 
McMaster University \\
Hamilton, ON L8S4L8, CA}
\email{macdos55@mcmaster.ca}

\begin{document}

\begin{abstract}
In this work we study global boundedness and exponential integrability of weak solutions to degenerate $p$-Poisson equations using an iterative method of De Giorgi type.  Given a symmetric, non-negative definite matrix valued function $Q$ defined on a bounded domain $\Omega\Subset\mathbb{R}^n$, a weight function $v\in L^1_\textrm{loc}(\Omega,dx)$, and a suitable non-negative function $\tau$, we give sufficient conditions for any weak solution to the Dirichlet problem
\begin{align*}
    \left\{\begin{array}{rccl}
    -\displaystyle\frac{1}{v}\mathrm{{div}}\left(\left|\sqrt{Q}\nabla u\right|^{p-2}Q\nabla u\right)+\tau\left|u\right|^{p-2}u&=&f&\textrm{in }\Omega,\\
    u&= & 0&\textrm{on }\partial\Omega
    \end{array}\right.
\end{align*}
to be bounded and exponentially integrable when the data function $f$ belongs to an appropriate Orlicz space. %In addition to the \(p\)-Laplace equation, we also study a much larger class of second order differential operators induced by convex functions.
\end{abstract}


\maketitle

\section{Introduction}

In a bounded domain $\Omega\Subset\mathbb{R}^n$, the authors of \cite{DCU-SR20} investigated \emph{a priori} boundedness of weak solutions to a class of Dirichlet problems for linear second order operators that includes
\begin{align*}
\left\{\begin{array}{rcl}
-\Delta u(x) &=& f(x)\textrm{ for }x\in \Omega\\
u(x)&=&0~~~\textrm{ for }x\in\partial\Omega,
\end{array}
\right.
\end{align*}
where the data function $f$ is permitted to be in the Orlicz class $L^A(\Omega,dx)$ with $A(t) = t^{n/2}\log(e-1+t)^{q}$ for $q>0$.  The main result of that paper shows that a weak solution to such a problem is bounded on $\Omega$ whenever $q>n/2$.   This paper continues the study of such problems.  Indeed, we investigate \emph{a priori} boundedness and exponential integrability of weak solutions to Dirichlet problems for operators of $p$-Laplace type with rough coefficients.  As in \cite{DCU-SR20}, the aim of this note is to explore optimal conditions on data functions that ensure regularity of the corresponding weak solution.  Our approach uses De Giorgi type iterative techniques to prove our results, and they and are modelled on the techniques used in \cite{DCU-SR20}. To effectively state the problems we study, we first describe the relevant class of operators.

Letting $S_n$ denote the collection of $n\times n$ symmetric, non-negative definite matrices, we let $Q :\Omega\rightarrow S_n$ be a matrix-valued function with real measurable entries. The pointwise operator norm of $Q$ is henceforth denoted by \(\omega:\Omega\rightarrow\mathbb{R}\). That is, for \(x\in\Omega\) we write \(\omega(x) =\|Q(x)\|_{\textrm{op}} =\sup_{|\xi|=1}\left|Q(x)\xi\right|\).  For $p>1$, let $v$ be a weight function in $\Omega$ that satisfies
\begin{align}\label{weightcond}v(x)\geq C\omega^{p/2}(x),\textrm{ with}~\omega\in L^{p/2}_\textrm{loc}(\Omega,dx).
\end{align}
More, we will always assume that the matrix function $Q$ satisfies a Sobolev inequality with gain $\sigma>1$. In particular, we assume there is are constants $C>0$ and $\sigma>1$ such that
\begin{equation}\label{sobolev}
    \left(\int_\Omega \left|\varphi \right|^{p\sigma}~vdx\right)^{1/p\sigma}\leq C\left(\int_\Omega |\sqrt{Q}\nabla \varphi |^p~dx\right)^{1/p}
\end{equation}
holds for any $\varphi\in \mathrm{Lip}_0(\Omega)$. The dual exponent of $\sigma$ plays a significant role in our results, and we denote it by $\sigma'$ at all times.

Given $\tau:\Omega\rightarrow [0,\infty]$, we consider Dirichlet problems of the form
\begin{align}\label{mainprob}
    \left\{\begin{array}{rccl}
    -\displaystyle\frac{1}{v}\mathrm{{div}} (| \sqrt{Q}\nabla u|^{p-2}Q\nabla u) + \tau\left|u\right|^{p-2}u&=&f&\textrm{in }\Omega\\
    u&= & 0&\textrm{on }\partial\Omega,
    \end{array}\right.
\end{align}
where the data function $f$ is in a $v$-weighted Orlicz class.  Our main results are the following two theorems.

\begin{thm}\label{main} Let $1< p<\infty$ and $\sigma>1$ be as in \eqref{sobolev}, and assume that $v(\Omega)<\infty$. Let $\tau\in L^{p\sigma'}(\Omega,v)$ be a non-negative, real valued function on $\Omega$, and let $f\in L^\Gamma(\Omega,v)$ for a Young function $\Gamma(t)$ for which \(t^{-\sigma'}\Gamma(t)\) is increasing on $[0,\infty)$ and 
\begin{align}\label{datacond}
    \int_1^\infty \frac{\Gamma'(t)}{\Gamma(t)}\left(\frac{t}{\Gamma(t)^\frac{1}{\sigma'}}\right)^\frac{1}{p-1}~dt<\infty.
\end{align}
Then, any weak sub-solution $(u,\nabla u)\in QH^{1,p}_0(\Omega,v)$ of \eqref{mainprob} is bounded satisfying
\[
    \|u\|_{L^\infty(\Omega,v)}\leq C\|f\|_{L^\Gamma(\Omega,v)}^\frac{1}{p-1}
\]
for a constant $C$ independent of both $(u,\nabla u)$ and $f$.
\end{thm}

\begin{thm}\label{expint} Let \(f\in L^{\sigma'}(\Omega,v)\), and let \(u\in QH_0^{1,p}(\Omega,v)\) be a bounded weak solution to \eqref{mainprob} with $\tau=0$. There is a constant $\gamma>0$ and a constant \(C>0\) such that
\[
    \int_\Omega e^{\gamma u}v dx\leq Cv(\Omega).
\]
We note that both $\gamma$ and $C$ are independent of the weak solution $(u,\nabla u)$.  
\end{thm}

%Given a Young function $\Psi$, a real-valued weight $v$ defined on $\Omega$, and a matrix-valued function \(Q\) as above, the principal part of the divergence operator we consider is formally defined by
%\[%Wanted to see how this looks without all the (x)'s
 %   Xu  = \displaystyle-\frac{1}{v}\mathrm{div}\Bigg(\frac{\Psi\big(\big|\sqrt{Q}\nabla u\big|\big)}{\big|\sqrt{Q}\nabla u\big|^2}~ Q\nabla u\Bigg).
%\]%\[Xu(x)  = \displaystyle-\frac{1}{v(x)}\mathrm{div}\Bigg(\frac{\Psi\big(\big|\sqrt{Q(x)}\nabla u(x)\big|\big)}{\big|\sqrt{Q(x)}\nabla u(x)\big|^2}~ Q(x)\nabla u(x)\Bigg).\]
\begin{rem}\label{introremark} 
\begin{itemize}[align=left,leftmargin=*,itemindent=0pt,labelsep=-5pt]
    \item[(1)]\;\;\;The existence of weak solutions to \eqref{mainprob} has been studied extensively; see for example, \cite{CRR2,CRR,HLWK,MRexist,SRexist}.  Weak solutions to problems similar to \eqref{mainprob} exist in the degenerate matrix weighted Sobolev spaces defined as closures of Lipschitz functions with respect to the Sobolev norm as described in \S2.2. Notation differs across these works; note for example that  $QH^{1,p}_0(\Omega,v)$ as in Theorem \ref{main} is described in Chapter 4 of \cite{CRR}, where the authors use the notation $\hat{H}^{1,p}_{Q,0}(\Omega)$.
    \item[(2)]\;\;\;Elements of the matrix weighted Sobolev space $QH^{1,p}_0(\Omega,v)$ are pairs of the form $(u,\vec{g}) \in L^p(\Omega,v)\times \mathcal{L}^p(\Omega,Q)$, where $\mathcal{L}^p(\Omega,Q)$ is the Banach space of $\mathbb{R}^n$-valued measurable functions as described in \cite{CRR2}. The function $\vec{g}$ is a generalized derivative of $u$ in the sense that there exists a sequence $\{u_j\}$ of Lipschitz functions with compact support in $\Omega$ such that
    \[
        \lim_{j\rightarrow\infty} \|u- u_j\|_{L^p(\Omega,v)} = 0\quad\textrm{and}\quad\lim_{j\rightarrow\infty} \|\sqrt{Q}(\vec{g} - \nabla u_j)\|_{\mathcal{L}^p(\Omega,Q)} = 0.
    \]
    We note that completeness of the matrix weighted space $\mathcal{L}^p(\Omega,Q)$ was first established for the case $p=2$ in \cite{SW2}, and for general $p\in[1,\infty)$ by the authors of \cite{CRR}.
    \item[(3)]\;\;\;A classical result of Gilbarg and Trudinger \cite[Thm. 8.16]{gt} is recovered in the case \(p=2\) by setting $Q = I$, taking $\sigma=\frac{n}{n-2}$, and using $\Gamma=t^q$ for $q>n/2$.  We note that \eqref{datacond} holds exactly when the well-known condition $q>\frac{n}{2}$ is satisfied.
\end{itemize}
\end{rem}

Theorem \ref{main}  generalizes \cite[Thm. 1.3]{DCU-SR20} in several ways. First, \cite[Thm. 1.3]{DCU-SR20} is a special case of Theorem \ref{main} obtained by taking $p=2$, $\tau=0$, with Young function $\Gamma(t)=t^{\sigma'}\log(e-1+t)^q$ for $q>\sigma'$ with $\sigma>1$. Condition \eqref{datacond} is satisfied since
\[
    \int_1^\infty  \frac{t\Gamma'(t)}{\Gamma(t)^{1+\frac{1}{\sigma'}}} \leq (q+\sigma')\int_{1}^\infty \frac{1}{t\log(e-1+t)^\frac{q}{\sigma'}} ~dt,
\]
and the integral on the right-hand side converges precisely when \(q>\sigma'\). Second, our method can accommodate the presence of the first-order drift term $\tau |u|^{p-2}u$ in \eqref{mainprob}.  This feature is useful and related to the existence of Sobolev and Poincar\'e inequalities, see \cite{CRR2,CRR}.  Lastly, we do not assume a special form of the Young function $\Gamma$ in replacing  \(\Gamma(t)=t^{\sigma'}\log(e-1+t)^q\) with any Young function satisfying  \eqref{datacond}. For example, our results are easily shown to apply when $\Gamma(t)$ is an iterated log bump Young function of the form
\begin{equation}\label{iterated}
    \Gamma(t)=\bigg(t\prod_{j=1}^{k-1}\underbrace{\log\cdots\log}_{j\;\mathrm{times}}(c_j+t)\bigg)^{\sigma'}\underbrace{\log\cdots\log}_{k\;\mathrm{times}}(c_k+t)^q,
\end{equation}
where $q>\sigma'$ and \(c_1,\dots,c_k\) are chosen independent of \(q\) and \(\sigma\) so that \(\Gamma(1)=1\).

The plan for the rest of this paper is as follows. In Section \ref{prelims} we recall the relevant notions of Orlicz spaces and weak solutions, and we discuss special properties of a family of test functions specially related to our iterative technique. Section \ref{mainproof} is devoted to the proof of Theorem \ref{main}. Following this, Section \ref{expintsec} discusses exponential integrability of weak solutions to \eqref{mainprob} and contains the proof of Theorem \ref{expint}.

\section{Preliminaries}\label{prelims}

This section provides a brief overview of the theory of Orlicz spaces needed for our results. Comprehensive treatments of this theory can be found in many of the works cited in this section.

\subsection{Orlicz Spaces}\; First we introduce the Young functions which induce Orlicz spaces. Such functions are commonly defined with respect to a density, which is any non-decreasing and right-continuous function \(\psi:[0,\infty)\rightarrow[0,\infty)\) which satisfies \(\psi(t)=0\) if and only if \(t=0\), and \(\psi(t)\rightarrow\infty\) as \(t\rightarrow\infty\). Given a density, the associated Young function \(\Psi\) is defined for \(t\geq 0\) by setting
\[
    \Psi(t) = \int_0^t\psi(s)ds.
\]
From this formulation, one sees that \(\Psi\) is continuous, strictly increasing, and convex on \((0,\infty)\). Moreover, it is clear that \(\Psi(0)=0\) and \(\Psi(t)\rightarrow\infty\) as \(t\rightarrow\infty\).

Equipped with a Young function $\Psi$ and a non-negative function $v\in L^1_\textrm{loc}(\Omega,dx)$ which we call a weight, one defines the Orlicz space $L^\Psi(\Omega,v)$ associated with a domain $\Omega\subset\mathbb{R}^n$ to be the collection of real-valued functions on \(\Omega\) for which the norm
\begin{equation}\label{lux} 
    \|f\|_{L^{\Psi}(\Omega,v)} = \inf\left\{\lambda>0~:~ \int_\Omega \Psi\left( \frac{|f(t)|}{\lambda}\right)~v(t)dt \leq 1\right\}
\end{equation}
is well-defined and finite. It is well-known that $L^\Psi(\Omega,v)$ is a Banach space; see \cite{MR924157}. Orlicz spaces generalize the classical Lebesgue spaces, and they are useful since they provide a finer scale of norms in the following sense: if \(\Psi(t)=t^p\log(e-1+t)^q\) for \(p\geq 1\) and \(q>0\), and if \(v\) is any non-vanishing weight, then for any \(\varepsilon>0\) we have
\[
    L^{p+\varepsilon}(\Omega,v)\subsetneq L^{\Psi}(\Omega,v)\subsetneq L^p(\Omega,v).
\]
The strictness of these inclusions follows from the examples in \cite{addie}. One can thus think of the `log-bump' space \(L^\Psi(\Omega,v)\) as an intermediate between Lebesgue spaces.

Just as the space $L^p(\Omega,v)$ has a topological dual which is isomorphic to \(L^\frac{p}{p-1}(\Omega,v)\), one can identify a dual Orlicz space to $L^\Psi(\Omega,v)$. Given a Young function \(\Psi\), the conjugate Young function is defined via the Legendre transform
\[
    \overline{\Psi}(t)=\sup_{s\geq 0}\{st-\Psi(s)\},
\]
and the Orlicz space \(L^{\overline{\Psi}}(\Omega,v)\) is (isomorphic to) the dual of \(L^{\Psi}(\Omega,v)\). This fact can be verified using the properties summarized in the following lemma. For proof of these properties, the reader is referred to \cite{rao-ren91}.

\begin{lem}\label{Young}
If $\Psi$ and $\overline{\Psi}$ are conjugate Young functions, then the following hold:
\begin{itemize}
    \item[(1)] Young's inequality: If \(s\geq0\) and \(t\geq0\) then $st\leq \Psi(s) + \overline{\Psi}(t)$,
    \item[(2)] Conjugate equivalence: If \(t\geq0\) then \(t\leq \Psi^{-1}(t)\overline{\Psi}^{-1}(t) \leq 2t\),
    \item[(3)] H\"older's Inequality: If $f\in L^\Psi(\Omega,v)$ and $g\in L^{\overline{\Psi}}(\Omega,v)$ then
    \[
    \bigg|\int_\Omega fg\;vdx\bigg| \leq 2\left\| f \right\|_{L^\Psi(\Omega,v)} \left\| g\right\|_{L^{\overline{\Psi}}(\Omega,v)}.
    \]
\end{itemize}
\end{lem}

It is useful to note that like H\"older's inequality, many other well-known properties of the Lebesgue spaces extend to Orlicz spaces in a natural way. One of these is Chebyshev's inequality, which generalizes in the following way when we replace the usual \(L^p\) norm with an Orlicz norm.

\begin{thm}[Chebyshev's Inequality]\label{cheby}
If \(f\in L^\Psi(\Omega,v)\) then for any \(\alpha>0\),
\begin{equation}\label{Markov}
    \Psi^{-1}(v(\{x\in\Omega:|f(x)|\geq\alpha\})^{-1})^{-1}\leq\frac{1}{\alpha}\|f\|_{L^\Psi(\Omega,v)}.
\end{equation}
\end{thm}

\begin{proof}
Verification of this inequality in \(L^1(\Omega,v)\) is an exercise common to most graduate measure theory courses, and we take for granted that for any \(f\in L^1(\Omega,v)\),
\[
    v(\{x\in\Omega:|f(x)|\geq\alpha\})\leq \frac{1}{\alpha}\int_\Omega|f|vdx.
\]
To establish \eqref{Markov} using the inequality above, we observe that it suffices to prove the estimate when \(\|f\|_{L^\Psi(\Omega,v)}=1\). Since \(\Psi\) is increasing and \(\Psi(|f|)\in L^1(\Omega,v)\), we have
\[
    v(\{x\in \Omega:|f(x)|\geq\alpha\})=v(\{x\in \Omega:\Psi(|f(x)|)\geq\Psi(\alpha)\})\leq\frac{1}{\Psi(\alpha)}\int_\Omega\Psi(|f|)vdx.
\]
Owing to our assumption that \(\|f\|_{L^\Psi(\Omega,v)}=1\), we find that the integral above is at most one. Consequently we have \(v(\{x\in \Omega:|f(x)|\geq\alpha\})\leq \Psi(\alpha)^{-1}\), and using the fact that \(\Psi^{-1}\) is increasing it is straightforward to rearrange and obtain \eqref{Markov}.
\end{proof}

Using the properties of Orlicz norms established above, one can show that if \(\Psi\) is a Young function and if \(\chi_S\) denotes the indicator of a measurable set \(S\subseteq\Omega\), then
\begin{equation}\label{characteristic}
    \|\chi_S\|_{L^\Psi(\Omega,v)}=\Psi^{-1}((v(S)^{-1})^{-1}.
\end{equation}
The final property of Orlicz norms which we require in the sections which follow is a generalized version of H\"older's inequality.

\begin{lem}\label{orliczinterp}
Let \(\Phi\), \(\Psi_1\), and \(\Psi_2\) be Young functions and assume  for each \(t\geq0\) that \(t\geq \Phi(\Psi_1^{-1}(t)\Psi_2^{-1}(t))\). If \(f\in L^{\Psi_1}(\Omega,v)\) and \(g\in L^{\Psi_2}(\Omega,v)\), then 
\[
    \|fg\|_{L^\Phi(\Omega,v)}\leq 2\|f\|_{L^{\Psi_1}(\Omega,v)}\|g\|_{L^{\Psi_2}(\Omega,v)}.
\]
\end{lem}

\begin{proof}
Once again, we observe that it suffices to prove the claimed estimate when 
\[
    \|f\|_{L^{\Psi_1}(\Omega,v)}=\|g\|_{L^{\Psi_2}(\Omega,v)}=1.
\]
Define \(\Omega_1=\{x\in\Omega:\Psi_1(|f(x)|)\geq \Psi_2(|g(x)|)\}\) and set \(\Omega_2=\Omega\setminus\Omega_1\). If \(x\in\Omega_1\) then 
\[
    \Phi(|fg|)\leq \Phi(\Psi_1^{-1}(\Psi_1(|f|))\Psi_2^{-1}(\Psi_1(|f|)))\leq\Psi_1(|f|),
\]
while if \(x\in\Omega_2\) then \(\Phi(|fg|)\leq \Psi_2(|g|)\). Since \(\Omega\subseteq\Omega_1\cup\Omega_2\) and \(\|f\|_{\Psi_1}=\|g\|_{\Psi_2}=1\), 
\[
    \int_\Omega\Phi(|fg|)vdx\leq\int_{\Omega_1}\Psi_1(|f|)vdx+\int_{\Omega_2}\Psi_2(|g|)vdx\leq  2.
\]
Finally, as \(\Phi\) is a Young function it is convex, meaning that \(\|fg\|_\Phi\leq 2\) by \eqref{lux} since
\[
    \int_\Omega\Phi\bigg(\frac{|fg|}{2}\bigg)vdx\leq \frac{1}{2}\int_\Omega\Phi(|fg|)vdx\leq 1.
\]
The claimed norm inequality then follows by dehomogenizing.
\end{proof}

Generalizing the proof above, one can show that if \(\Psi_1^{-1}\cdots \Psi_m^{-1}\leq \Phi^{-1}\) pointwise, and if \(f_j\in L^{\Psi_j}(\Omega,v)\) for \(j=1,\dots,m\), then
\[
    \|f_1\cdots f_m\|_{L^\Phi(\Omega,v)}\leq m\|f_1\|_{L^{\Psi_1}(\Omega,v)}\cdots \|f_m\|_{L^{\Psi_m}(\Omega,v)}.
\]
Orlicz spaces are well-studied, and they enjoy many more properties analogous to the Lebesgue spaces. For our purposes though, the brief outline given above suffices.

\subsection{Sobolev Spaces and Test Functions}\; Fix $1\leq p<\infty$ and let \(Q\) be an \(n\times n\) symmetric, non-negative definite matrix-valued function with pointwise operator norm $\omega$.
%Since the wieght $v$ satisfies $w^{p/2}\leq Cv$ a.e. The pointwise operator norm $\omega$ controls our weight $v$ from below, since we assume that Hypothesis \ref{admissibility} holds. Owing to this control, for each  \(\xi\in\mathbb{R}^n\) we have the pointwise upper-ellipticity estimate
%\[
 %   |\sqrt{Q}~\xi|^p \leq \omega^{p/2} | \xi|^p\leq v\left|\xi\right|^p.
%\]

%From \cite[Lemma 2.3.1]{RS} it follows that there exists a \(v\)-measurable unitary matrix-valued function $U$ and a diagonal matrix-valued function \(D\) such that \(Q = U^tDU\) almost everywhere in \(\Omega\). Thus, there exists a well-defined `square root' of \(Q\) of the form \(\sqrt{Q} = U^t\sqrt{D}U\), where $\sqrt{D}$ is the matrix whose entries are the square roots of those in $D$. Conveniently, this allows us to write \(\xi^tQ\xi = |\sqrt{Q}~\xi|^2\) for any $\xi\in \mathbb{R}^n$. 

The \(Q\)-weighted vector-valued Lebesgue space \(\mathcal{L}^p(\Omega,Q)\) is defined as the set of all functions \(\vec{g}:\Omega\rightarrow\mathbb{R}^n\) for which the norm
\[
    \|\vec g\|_{\mathcal{L}^p(\Omega,Q)}=\bigg(\int_\Omega|\sqrt{Q}~\vec{g}|^pdx\bigg)^\frac{1}{p}<\infty.
\]
Identifying any two measurable $\mathbb{R}^n$-valued functions $\vec{f}$ and $\vec{g}$ if $\|\vec{f}-\vec{g}\|_{\mathcal{L}^p(\Omega,Q)}=0$, it is established in \cite{CRR2} that $\mathcal{L}^p(\Omega,Q)$ is a Banach space. Using this norm, the Sobolev space \(QH^{1,p}_0(\Omega,v)\) is defined as the completion of \(\mathrm{Lip}_0(\Omega)\) with respect to the norm
\begin{align}\label{sobonormnew}
    \|u\|_{QH^{1,p}(\Omega,v)}=\|u\|_{L^p(\Omega,v)}+\|\nabla u\|_{\mathcal{L}^p(\Omega,Q)}.
\end{align}
In other words, \(QH^{1,p}(\Omega,v)\) is the collection of ordered pairs \((u,\vec{g})\in L^p(\Omega,v)\times \mathcal{L}^p(\Omega,Q)\) for which there exists a sequence \(\{u_j\}\) of Lipschitz functions with compact support in $\Omega$ that satisfies
\begin{equation}\label{limid}
    \lim_{j\rightarrow\infty}\|u-u_j\|_{L^p(\Omega,v)}=0\qquad\mathrm{and}\qquad \lim_{j\rightarrow\infty}\|\vec{g}-\nabla u_j\|_{\mathcal{L}^p(\Omega,Q)}=0.
\end{equation}
The space \(QH^{1,p}(\Omega,v)\) is defined analogously via completion of the collection of $\mathrm{Lip}(\Omega)$ functions with respect to \eqref{sobonormnew}.

In some settings, owing to degeneracy of \(Q\), the function \(\vec{g}\) need not be uniquely determined by \(u\). An example of this behaviour is given in \cite{FKS}. Nevertheless, we often abuse notation by writing \(\nabla u\) in place of \(\vec{g}\) and referring to \(\nabla u\) as a formal gradient of \(u\), but we caution the reader that this identification is made only in the sense of \eqref{limid}. More, to simplify notation, we often write \(u\in QH^{1,p}_0(\Omega,v)\) in place of $(u,\nabla u)\in QH^{1,p}_0(\Omega,v) $.

%Our main assumption concerning the matrix function $Q$ and the space \(QH^{1,p}_0(\Omega,v)\) is the following weighted Sobolev inequality. 

%\begin{hyp}\label{sobprop} The matrix function $Q$ is said to satisfy a Sobolev property of order $p>1$ on $\Omega$ with gain $\sigma>1$, if there is a constant $C>0$ such that that 
%\begin{equation}\label{sobolev}
 %   \left(\int_\Omega \left|\varphi \right|^{p\sigma}~vdx\right)^{1/p\sigma}\leq C\left(\int_\Omega |\sqrt{Q}\nabla \varphi |^p~dx\right)^{1/p}
%\end{equation}
%for every \(\varphi\in\mathrm{Lip}_0(\Omega)\), the space of Lipschitz functions compactly supported in $\Omega$.
%\end{hyp}

%Inequality \eqref{sobolev} can be written compactly as \(\|\varphi\|_{L^{p\sigma}(\Omega,v)}\leq \|\sqrt{Q}\nabla\varphi\|_{\mathcal{L}^{p}(\Omega,dx)}\). Much literature exists concerning the validity of inequalities  like \eqref{sobolev}. As a starting point, the interesting reader is referred to XXXX[Sawyer/Wheeden] and the references therein.

\subsubsection{Weak Solutions}\; Now we give a definition of weak solution of the Dirichlet problem \eqref{mainprob}.  The conditions required on the data function $f$ and non-negative function $\tau$ ensure that the integrals in the definition are finite.% Given a non-negative function $\tau\in L^{p\sigma'}(\Omega,v)$, where \(\sigma'\) is the H\"older conjugate of \(\sigma\), we now consider the notion of weak solution to the Dirichlet problem \eqref{mainprob}. When the data function \(f\) belongs to $ L^{p'}(\Omega,v)$ and $v\equiv1$, it is established in \cite[Proposition 4.4]{CRR} that weak solutions exist in \(QH_0^{1,p}(\Omega)\).

%In general, our assumptions on \(f\) will be weaker than those assumed in the cited literature on existence, and in this work we do not claim that solutions to the equations we study necessarily exist. When \(Q\) is very degenerate much of the classical theory for elliptic equations breaks down, and standard techniques for verifying existence and uniqueness of solutions can fail. In any case, we define weak solutions to \eqref{mainprob} in an analogous way to the elliptic setting. 

\begin{defn}\label{weaksol} 
Let $Q$ and the weight $v$ satisfy \eqref{weightcond} and \eqref{sobolev}. Given a function $f\in L^{(p\sigma)'}(\Omega,v)$ and non-negative $\tau\in L^{\sigma'}(\Omega,v)$, a pair $(u,\vec{g})\in QH^{1,p}_0(\Omega,v)$ is called a weak solution to \eqref{mainprob} if the identity 
\begin{align}\label{weaks}
    \int_\Omega |\sqrt{Q}~\vec{g} |^{p-2} \nabla \varphi \cdot Q~\vec{g}~dx + \int_\Omega \tau \left|u \right|^{p-2}u\varphi ~vdx = \int_\Omega f\varphi ~vdx.
\end{align}
holds for every $\varphi \in \mathrm{Lip}_0(\Omega)$.
\end{defn}

Our notion of weak solution is well-defined since the integrals in \eqref{weaks} converge by H\"older's inequality and our assumptions on \(f\), \(\tau\), \(u\), and \(\vec{g}\). Using a simple density argument together with Fatou's lemma, we can expand our test functions from \(\mathrm{Lip}_0(\Omega)\) to include all member pairs of $QH^{1,p}_0(\Omega,v)$. That is, given any pair $(w,\vec{h})\in QH^{1,p}_0(\Omega,v)$ one has that
\begin{align}\label{weaksgen}
    \int_\Omega \left|\sqrt{Q}~\vec{g}\right|^{p-2} \vec{h} \cdot Q\vec{g}~dx + \int_\Omega  \tau\left|u \right|^{p-2}u w ~vdx = \int_\Omega f w ~vdx.
\end{align}
%As above we caution that \(\nabla u\) and \(\nabla\varphi\) are formal gradients in the sense of \eqref{limid}.

\subsection{Technical Results} In Sections \ref{mainproof} and \ref{expintsec}, given a weak solution $u\in QH^{1,p}_0(\Omega,v)$ we will use test functions \(\varphi_\alpha = (u-\alpha)_+\) and \(w=e^{\alpha u}-1\) for arbitrary \(\alpha>0\). The following technical lemmas demonstrate that these are valid test functions that can be used in \eqref{weaksgen}. %To prove Lemma \ref{testcomp}, we require Proposition \ref{positiveapproximator}, which generalizes a similar result appearing in \cite{DCU-SR20}. 

\begin{lem}\label{testfunction}
Let $Q$ be as above, let inequality \eqref{weightcond} hold, and fix $\alpha>0$. Given a pair $(u,\nabla u )\in QH^{1,p}_0(\Omega,v)$, set $\varphi_\alpha  = (u-\alpha)_+$. Then \((\varphi_\alpha,\chi_{\{u\geq\alpha\}}\nabla u ) \in QH^{1,p}_0(\Omega,v)\).
\end{lem}

\begin{proof} 
Recalling the definition of \(QH^{1,p}_0(\Omega,v)\), we see that it is enough to identify a sequence $\{\psi_j\}$ in $\mathrm{Lip}_0(\Omega)$ for which
\[
    \lim_{j\rightarrow\infty}\|\psi_j-\varphi_\alpha\|_{L^p(\Omega,v)} = 0 \qquad\textrm{and}\qquad \lim_{j\rightarrow\infty}\|\nabla \psi_j -\chi_{\{u\geq\alpha\}}\nabla u \|_{\mathcal{L}^p(\Omega,Q)}=0.
\]
Since $u\in QH^{1,p}_0(\Omega,v)$, there is a sequence $\{u_j\}$ of \(\mathrm{Lip}_0(\Omega)\) functions that converge to $u$ in $L^p(\Omega,v)$, and whose gradients converge to the formal gradient $\nabla u $ in $\mathcal{L}^p(\Omega,Q)$. Passing to a subsequence if necessary, we may assume that $u_j\rightarrow u$ \(v\)-a.e. and that $\sqrt{Q}\nabla u_j \rightarrow\sqrt{Q}\nabla u $ a.e. in $\Omega$. For each $j\in\mathbb{N}$, define $\psi_j = (u_j-\alpha)_+$ so that pointwise $v$-a.e. we have $\psi_j\rightarrow \varphi_\alpha$. In addition,
\[
    | \psi_j-\varphi_\alpha|^p\leq |u-u_j|^p\leq 2^{p-1}(|u|^p+|u_j|^p)
\]
holds \(v\)-a.e., and as the right-hand side converges pointwise $v$-a.e. to \(2^p|u|^p\in L^1(\Omega,v)\) it follows from the dominated convergence theorem that $\psi_j\rightarrow \varphi_\alpha$ in $L^p(\Omega,v)$ also.

To see the convergence of gradients, we note that \(\nabla\psi_j=\chi_{\{u_j\geq \alpha\}}\nabla u_j\) for each \(j\in\mathbb{N}\) by \cite[Lemma 7.6]{gt}, and that $\chi_{\{u_j\geq\alpha\}}\rightarrow \chi_{\{u\geq\alpha\}}$ pointwise $v$-a.e. in \(\Omega\). Since
\[
    |\nabla \psi_i-\chi_{\{u\geq\alpha\}}\nabla u |\leq |(\chi_{\{u_j\geq\alpha\}}- \chi_{\{u\geq\alpha\}})\nabla u  | + |\chi_{\{u_j\geq\alpha\}}(\nabla u_j -\nabla u )|,
\]
another use of dominated convergence shows that \(\nabla \psi_i-\chi_{\{u\geq\alpha\}}\nabla u\) converges to zero in $\mathcal{L}^p(\Omega,Q)$, meaning that $(\varphi_\alpha,\nabla u \chi_{S_\alpha})\in QH^{1,p}_0(\Omega,v)$ as claimed.
\end{proof}

\begin{pro}\label{positiveapproximator}
Let $(u,\nabla u)\in QH^{1,p}_0(\Omega,v)$ with $u\geq 0$ $v$-a.e. in $\Omega$, and assume that \(u\) is essentially bounded. Then there exists a sequence $\{u_j\}$ of $\mathrm{Lip}_0(\Omega)$ functions satisfying the following criteria:
\begin{enumerate}
    \item $0\leq u_j(x) \leq \|u\|_{L^\infty(\Omega,v)}+1$ $v$-a.e. in $\Omega$.
    \item $u_j\rightarrow u$ both pointwise $v$-a.e. in \(\Omega\) and in $L^p(\Omega,v)$.
    \item $\nabla u_j \rightarrow \nabla u $ in $\mathcal{L}^p(\Omega,Q)$ and $\sqrt{Q}\nabla u_j \rightarrow \sqrt{Q}\nabla u $ $v$-a.e. in $\Omega$.
    \item $\|\nabla u_j\|_{\mathcal{L}^p(\Omega,Q)}\leq \| \nabla u \|_{\mathcal{L}^p(\Omega,Q)}+1$ for all $j\in\mathbb{N}$.
\end{enumerate}
\end{pro}

\begin{proof}
Let $\{\psi_j\}$ be a sequence of $\mathrm{Lip}_0(\Omega)$ functions that converge to $u$ in the norm of $QH^{1,p}_0(\Omega,v)$. Passing to subsequences if necessary, we may assume that $\psi_j\rightarrow u$ holds pointwise $v$-a.e. and in $L^p(\Omega,v)$, that $\nabla \psi_j\rightarrow \nabla u $ in $\mathcal{L}^p(\Omega,Q)$ and $\sqrt{Q}\nabla \psi_j\rightarrow \sqrt{Q}\nabla u $ pointwise a.e. in $\Omega$, and finally that $\|\nabla \psi_j\|_{\mathcal{L}^p(\Omega,Q)}\leq \| \nabla u \|_{\mathcal{L}^p(\Omega,Q)}+1$ for each $j\in\mathbb{N}$.

Following \cite{DCU-SR20}, we set $w_j=|\psi_j|$ so that \(|w_j-u|\leq |\psi_j-u|\) pointwise \(v\)-a.e. in \(\Omega\), from which it follows that $w_j\rightarrow u$ pointwise \(v\)-a.e. and in  $L^p(\Omega,v)$. Since $w_j$ is Lipschitz, it is differentiable \(v\)-a.e. by the Rademacher-Stepanov Theorem, and \(v\)-a.e. we also have \(\nabla w_j = \mathrm{sgn}(\psi_j)\nabla \psi_j\). This gives $|\sqrt{Q}\nabla w_j| = |\sqrt{Q}\nabla \psi_j|$ a.e., and we see that the sequence $\{w_j\}$ inherits the properties of \(\{\psi_j\}\) outlined above.

We construct the sequence $\{u_j\}$ as in \cite[Lemma 2.15]{DCU-SR20} by setting $M = \|u\|_{L^\infty(\Omega,v)}$ and letting $\phi\in C^1([0,\infty)$ be a non-negative increasing function for which \(\phi(t)=t\) when \(t\leq M-\frac{1}{2}\), \(\phi(t)=M\) for \(t\geq M+1\), and \(0\leq \phi'(t)\leq 1\) globally. Taking \(u_j = \phi(w_j)\), it is easy to verify that (1)-(4) hold, as is done in \cite[Lemma 2.15]{DCU-SR20}.
\end{proof}

\begin{lem}\label{testcomp}
Let \(u\in QH_0^{1,p}(\Omega,v)\) be non-negative and essentially bounded. Given \(\phi\in C^1(\mathbb{R})\) such that \(\phi(0)=0\), we have \((\phi(u),\phi'(u)\nabla u)\in  QH_0^{1,p}(\Omega)\).
\end{lem}

\begin{proof}
By Proposition \ref{positiveapproximator}, we may select a sequence $\{u_k\}$ of non-negative bounded $\mathrm{Lip}_0(\Omega)$ functions which converge to \(u\) in the norm of $QH^{1,p}_0(\Omega,v)$, and which satisfy properties (1) through (4). It is clear that $\phi(u_k)\in \mathrm{Lip}_0(\Omega)$ for each $k$ as $\phi\in C^1(\mathbb{R})$. Moreover, the continuity of \(\phi\) ensures that \(\phi(u_k)\) converges to \(\phi(u)\) $v$-a.e. in \(\Omega\).

Let \(M\) be such that \(u_j\leq M\) for all \(j\). Since \(\phi'\) is continuous it is bounded on \([0,M]\), and as \(\phi(0)=0\) there exists \(C>0\) such that \(|\phi(t)|=|\phi(t)-\phi(0)|\leq C|t|\) for \(t\leq M\). Therefore
\[
    |\phi(u_k)-\phi(u)|^p\leq 2^{p-1}(|\phi(u_k)|^p+|\phi(u)|^p)\leq 2^{p-1}C^p(|u_k|^p+|u|^p),
\]  
and since the right-hand side converges to \(2^pC^p|u|^p\) in \(L^1(\Omega,v)\), the dominated convergence theorem implies that \(\phi(u_k)\rightarrow\phi(u)\) in \(L^p(\Omega,v)\). The gradient converges in a similar manner; we have \(\sqrt{Q}\nabla\phi(u_k)=\phi'(u_k)\sqrt{Q}\nabla u_k\) almost everywhere, and continuity of \(\phi'\) thus ensures that \(\sqrt{Q}\nabla\phi(u_k)\rightarrow\phi'(u)\sqrt{Q}\nabla u\) pointwise a.e. in \(\Omega\). Additionally we have
\[
    |\sqrt{Q}\nabla\phi(u_k)-\phi'(u)\sqrt{Q}\nabla u|^p\leq 2^{p-1}C^p(|\sqrt{Q}\nabla u_k|^p+|\sqrt{Q}\nabla u|^p),
\]
and once again the right-hand side converges to \(2^pC^p|\sqrt{Q}\nabla u|^p\) in \(\mathcal{L}^1(\Omega,dx)\). The dominated convergence theorem thus lets us conclude that \(\sqrt{Q}\nabla\phi(u)\) converges to \(\phi'(u)\sqrt{Q}\nabla u\) in \(L^p(\Omega,v)\), giving \((\phi(u),\nabla\phi(u))\in QH_0^{1,p}(\Omega,v)\) as required.
\end{proof}

As a last note to this section, we show that under the hypotheses of Theorem \ref{main}, the weak formulation \eqref{weaks} of problem \eqref{mainprob} is well-defined. Recall, we assume that \(f\in L^\Gamma(\Omega,v)\) where \(\Gamma\) is a Young function for which \(t^{-\sigma'}\Gamma(t)\) is non-decreasing on \([0,\infty)\) and \eqref{datacond} holds. In contrast, \eqref{weaks} is well defined only under the assumption \(f\in L^{(p\sigma)'}(\Omega,v)\). We now show that the former condition implies the latter.

\begin{lem}
Suppose $v(\Omega)<\infty$, $1<p<\infty$, and let $\sigma,v$ be as in \eqref{weightcond} and \eqref{sobolev}. Let \(\Gamma\) be a Young function satisfying \eqref{datacond}. Then, \(L^\Gamma(\Omega,v)\subseteq L^{(p\sigma)'}(\Omega,v)\).  
\end{lem}

\begin{proof}
First we let \(q\) be the non-negative function for which \(q(t)\Gamma(t)=t^{(p\sigma)'}\), and we show that \(q\) must be bounded for \(t\geq 1\). Assume to the contrary that \(q(t)\rightarrow\infty\) as \(t\rightarrow\infty\), and observe that since \(\Gamma(t)\leq t\Gamma'(t)\) we have
\[
    \int_1^\infty \frac{\Gamma'(t)}{\Gamma(t)}\left(\frac{t}{\Gamma(t)^\frac{1}{\sigma'}}\right)^\frac{1}{p-1}dt\geq \int_1^\infty \frac{1}{t}\left(\frac{t}{\Gamma(t)^\frac{1}{\sigma'}}\right)^\frac{1}{p-1}dt=\int_1^\infty t^{\frac{1}{p\sigma-1}-1}q(t)^\frac{\sigma-1}{\sigma(p-1)}dt.
\]
For any \(k\geq 1\), it follows from the inequality \(p\sigma>1\) and the estimate above that 
\[
    \int_1^\infty \frac{\Gamma'(t)}{\Gamma(t)}\left(\frac{t}{\Gamma(t)^\frac{1}{\sigma'}}\right)^\frac{1}{p-1}dt\geq \int_k^{2k} t^{\frac{1}{p\sigma-1}-1}q(t)^\frac{\sigma-1}{\sigma(p-1)}dt\geq \log 2\bigg(\min_{t\geq k}q(t)^\frac{\sigma-1}{\sigma(p-1)}\bigg).
\]
Taking a limit in \(k\), we find that
\[
    \int_1^\infty \frac{\Gamma'(t)}{\Gamma(t)}\left(\frac{t}{\Gamma(t)^\frac{1}{\sigma'}}\right)^\frac{1}{p-1}dt\geq\log 2\bigg(\liminf_{k\rightarrow\infty}q(k)^\frac{\sigma-1}{\sigma(p-1)}\bigg)=\infty,
\]
which contradicts our assumption that the integral on the left-hand side converges. We conclude that \(q(t)\leq C\) for \(t\geq 1\), meaning that \(t^{(p\sigma)'}\leq C\Gamma(t)\) for \(t\geq 1\). For later convenience, we replace \(C\) with \(\max\{1,C\}\) to assume that \(C\geq 1\). 

Taking \(\lambda=C\max\{2,2v(\Omega)\}^\frac{1}{(p\sigma)'}\|f\|_\Gamma\) and making a straightforward estimate gives
\[
    \int_\Omega \bigg(\frac{|f|}{\lambda}\bigg)^{(p\sigma)'}vdx\leq\frac{v(\Omega\cap \{|f|\leq C\|f\|_\Gamma\})}{2v(\Omega)}+\frac{1}{2}\int_{\Omega\cap\{|f|>C\|f\|_\Gamma\}} \bigg(\frac{|f|}{C\|f\|_\Gamma}\bigg)^{(p\sigma)'}vdx.
\]
Thanks to the estimate \(t^{(p\sigma)'}\leq C\Gamma(t)\) for \(t\geq 1\), we can observe by convexity that
\[
    \int_{\Omega\cap\{|f|>C\|f\|_\Gamma\}} \bigg(\frac{|f|}{C\|f\|_\Gamma}\bigg)^{(p\sigma)'}vdx\leq C\int_{\Omega} \Gamma\bigg(\frac{|f|}{C\|f\|_\Gamma}\bigg)vdx\leq \int_{\Omega} \Gamma\bigg(\frac{|f|}{\|f\|_\Gamma}\bigg)vdx\leq 1.
\]
For \(\lambda\) as above then, we find that \(\|f\|_{(p\sigma)'}\leq\lambda\) since
\[
    \int_\Omega \bigg(\frac{|f|}{\lambda}\bigg)^{(p\sigma)'}vdx\leq\frac{v(\Omega\cap \{|f|\leq C\|f\|_\Gamma\})}{2v(\Omega)}+\frac{1}{2}\leq 1.
\]
If follows that if \(f\in L^\Gamma(\Omega,v)\) and if \(v(\Omega)<\infty\), then \(f\in L^{(p\sigma)'}(\Omega,v)\) as claimed.
\end{proof}

If \(\Gamma\) is as above and if \(v\) is essentially constant on \(\Omega\), the preceding lemma and \cite[Proposition 4.4]{CRR} together imply that \eqref{mainprob} has a weak solution in \(QH_0^{1,p}(\Omega,v)\). Existence of weak solutions to a similar class of homogeneous Neumann problems is also discussed in \cite[Chapter 4]{CRR2}, under more general hypotheses on the weight \(v\).

\section{Proof of Theorem \ref{main}}\label{mainproof}

Following standard convention, we use \(C\) to denote a constant which may change from line to line. Fix a Young function $\Gamma$ satisfying \eqref{datacond} for which \(t^{-\sigma'}\Gamma(t)\) is non-decreasing, let $u\in QH^{1,p}(\Omega,v)$ be a non-negative weak solution to \eqref{mainprob}, and assume without loss of generality that \(\Gamma(1)=1\) and \(v(\Omega)=1\). For any constant $\alpha>0$ we define $\varphi_\alpha = (u-\alpha)_+$, and we show that \(v_\alpha=v(\{x\in\Omega:u(x)\geq\alpha\})=0\) when \(\alpha\) is sufficiently large. From this identity it will follow that \(\|u\|_{L^\infty(\Omega,v)}\leq\alpha\).

For any \(\beta>\alpha\geq 0\), we first estimate $v_\beta$ in terms of $v_\alpha$ to obtain a recursion in the measures of level sets. Using Theorem \ref{cheby} and H\"older's inequality, we find that
\begin{equation}\label{cheb}
    (\beta-\alpha)v_\beta\leq \int_\Omega\varphi_\alpha vdx\leq \|\varphi_\alpha\|_{L^{p\sigma}(\Omega,v)}v_\alpha^{1-\frac{1}{p\sigma}}.    
\end{equation}
Our aim is to estimate $\|\varphi_\alpha\|_{L^{p\sigma}(\Omega,v)}$ in terms of \(v_\alpha\). Since $\varphi_\alpha\in QH^{1,p}_0(\Omega,v)$ by Lemma \ref{testfunction}, we can use it as a test function in \eqref{weaksol} to get
\[
    \int_\Omega|\sqrt{Q}\nabla u|^{p-2}\sqrt{Q}\nabla u\cdot\sqrt{Q}\nabla\varphi_\alpha dx+\int_\Omega\tau |u|^{p-2}u\varphi_\alpha vdx=\int_\Omega f\varphi_\alpha vdx.
\]
The definition of \(\varphi_\alpha\) gives \(|\sqrt{Q}\nabla u|^{p-2}\sqrt{Q}\nabla u\cdot\sqrt{Q}\nabla\varphi_\alpha=|\sqrt{Q}\nabla\varphi_\alpha|^p\), and moreover the second integral above is non-negative. It follows from \eqref{sobolev} that
\[
    \|\varphi_\alpha\|_{L^{p\sigma}(\Omega,v)}^p\leq C\|\sqrt{Q}\nabla\varphi_\alpha\|_{L^p(\Omega,dx)}^p\leq C\int_\Omega f\varphi_\alpha vdx.
\]

To proceed, we estimate the integral on the right-hand side above by employing H\"older's inequality once again to get
\[
    \int_\Omega f\varphi_\alpha vdx\leq 2\|f\|_{L^\Gamma(\Omega,v)}\|\varphi_\alpha\|_{L^{\overline{\Gamma}}(\Omega,v)}.
\]
If \(\Psi\) is the Young function that satisfies \(\Psi^{-1}(t)=\overline{\Gamma}^{-1}(t)t^{-\frac{1}{p\sigma}}\) for \(t\geq0\), then Lemma \ref{orliczinterp} gives us \(\|\varphi_\alpha\|_{L^{\overline{\Gamma}}(\Omega,v)}\leq 2\|\varphi_\alpha\|_{L^{p\sigma}(\Omega,v)}\|\chi_\alpha\|_{L^\Psi(\Omega,v)},\) where \(\chi_\alpha\) is the indicator of the set \(\{x\in\Omega:u(x)\geq\alpha\}\). From \eqref{characteristic} and item \textit{(2)} of Lemma \ref{Young} we also have
\[
     \|\chi_\alpha\|_{L^\Psi(\Omega,v)}=\Psi^{-1}(v_\alpha^{-1})^{-1}\leq  v_\alpha^{1-\frac{1}{p\sigma}}\Gamma^{-1}(v_\alpha^{-1}),
\]
meaning that
\[
    \int_\Omega f\varphi_\alpha vdx\leq 4\|f\|_{L^\Gamma(\Omega,v)}\|\varphi_\alpha\|_{L^{p\sigma}(\Omega,v)}v_\alpha^{1-\frac{1}{p\sigma}}\Gamma^{-1}(v_\alpha^{-1}).
\]
Combining our estimates involving \(f\) and \(\|\varphi_\alpha\|_{L^{p\sigma}(\Omega,v)}\) and simplifying now gives
\[
    \|\varphi_\alpha\|_{L^{p\sigma}(\Omega,v)}\leq C\|f\|_{L^\Gamma(\Omega,v)}^\frac{1}{p-1}v_\alpha^{\frac{p\sigma-1}{p\sigma(p-1)}}\Gamma^{-1}(v_\alpha^{-1})^\frac{1}{p-1}.
\]

Using \eqref{cheb} and the estimate above, we obtain the desired recursion in measure:
\[
    (\beta-\alpha)v_\beta\leq C\|f\|_{L^\Gamma(\Omega,v)}^\frac{1}{p-1}v_\alpha^{\frac{p\sigma-1}{p\sigma -\sigma}}\Gamma^{-1}(v_\alpha^{-1})^\frac{1}{p-1}.
\]
To simplify, we replace \(u\) with \(\Tilde{u}=u/C\rho \|f\|_{L^\Gamma(\Omega,v)}^\frac{1}{p-1}\) for \(C\) as above and \(\rho>1\) to get
\[
    (\beta-\alpha)v_\beta\leq \frac{1}{\rho}v_\alpha^{\frac{p\sigma-1}{p\sigma -\sigma}}\Gamma^{-1}(v_\alpha^{-1})^\frac{1}{p-1}.
\]
Exploiting this recursion, we show now that \(\Tilde{u}\) is essentially bounded.

For a monotone increasing sequence \(\{s_k\}\) to be chosen momentarily, we define \(v_k=\{x\in\Omega:u(x)\geq s_k\}\) so that the preceding inequality yields
\[
    (s_{k+1}-s_k)v_{k+1}\leq \rho^{-1}v_k^{\frac{p\sigma-1}{p\sigma -\sigma}}\Gamma^{-1}(v_k^{-1})^\frac{1}{p-1}.
\]
To construct \(\{s_k\}\) we set \(s_0=0\) and note that \(v_0\leq v(\Omega)=1\). Given \(s_k\) we let
\[
    s_{k+1}=s_k+v_k^{\frac{1}{(p-1)\sigma'}}\Gamma^{-1}(v_k^{-1})^\frac{1}{p-1}
\]
so that \(v_{k+1}\leq \rho^{-1}v_k\) and \(v_k\leq \rho^{-k}\). It follows that \(v_k\rightarrow0\) as \(k\rightarrow\infty\), and so
\[
    \|\Tilde{u}\|_{L^\infty(\Omega,v)}\leq \lim_{k\rightarrow\infty }s_k=\sum_{k=0}^\infty v_k^{\frac{1}{(p-1)\sigma'}}\Gamma^{-1}(v_k^{-1})^\frac{1}{p-1}.
\]
Since \(t^{-\sigma'}\Gamma(t)\) is non-decreasing, it is easy to show that \(t^\frac{1}{\sigma'}\Gamma^{-1}(t^{-1})\) is also non-decreasing, whereby we conclude from the estimate \(v_k\leq \rho^{-k}\) that
\[
    \|\Tilde{u}\|_{L^\infty(\Omega,v)}\leq \sum_{k=0}^\infty \rho^{-\frac{k}{(p-1)\sigma'}}\Gamma^{-1}(\rho^k)^\frac{1}{p-1}\leq 1+\int_0^\infty\rho^{-\frac{t}{(p-1)\sigma'}}\Gamma^{-1}(\rho^t)^\frac{1}{p-1}dt.
\]
Making the change of variables \(\rho(t)=\Gamma(s)\) and using condition \eqref{datacond}, we see that
\[
    \|\Tilde{u}\|_{L^\infty(\Omega,v)}\leq 1+\frac{1}{\log(\rho)}\int_1^\infty\frac{\Gamma'(s)}{\Gamma(s)}\bigg(\frac{s}{\Gamma(s)^{\frac{1}{\sigma'}}}\bigg)^\frac{1}{p-1}ds\leq C.
\]
It follows that \(\|u\|_{L^\infty(\Omega,v)}\leq C\|f\|_{L^\Gamma(\Omega,v)}^\frac{1}{p-1}\), as we wished to show.\hfill\(\Box\)

\section{Exponential Integrability}\label{expintsec}

In this section we study exponential integrability of weak solutions to the Dirichlet problem \eqref{mainprob} when $\tau=0$. That is, if \(u\) is a weak solution to the degenerate Poisson-type problem
\begin{equation}\label{notau}
    \left\{\begin{aligned}
        -\frac{1}{v(x)}\mathrm{div}(|\sqrt{Q(x)}\nabla u|^{p-2}Q(x)\nabla u(x))=\;&f(x)\quad &&x\in\Omega,\\
        \hfill u(x)=&\;0 \qquad &&x\in\partial\Omega,
    \end{aligned}\right.
\end{equation}
then  \(\|e^{\gamma u}\|_{L^1(\Omega,v)}\leq Cv(\Omega)\) when $v(\Omega)<\infty$ and \(\gamma\) is a small constant small enough. Our approach is modelled on \cite[Section 2]{DCU-SR20}, and we begin with the following norm estimate.

\begin{lem}\label{aux}
Let \(u\) be a non-negative bounded weak solution to \eqref{notau}, and for constants \(\xi>0\) and \(\alpha>0\) let \(w=e^{\xi u}-1\) and \(w_\alpha=(w-\alpha)_+\). Then
\begin{equation}\label{first}
    \|\sqrt{Q}\nabla w_\alpha\|_{L^p(\Omega,dx)}^p\leq \xi^{p-1}\int_\Omega |f|w_\alpha(w+1)^{p-1}vdx.
\end{equation}
\end{lem}

\begin{proof}
From Lemma \ref{testcomp} it follows that \(\psi_\alpha=((w+1)^p-(\alpha+1)^p)_+\in QH_0^1(\Omega,v)\), meaning that we can use \(\psi_\alpha\) in the weak formulation of \eqref{notau}. Further, observe that \(\sqrt{Q}\nabla\psi_\alpha=\chi_\alpha p\xi (w+1)^{p}\sqrt{Q}\nabla u\), where \(\chi_\alpha\) denotes the indicator function of the support of \(w_\alpha\). Consequently we have
\[
    \int_\Omega|\sqrt{Q}\nabla u|^{p-2}\sqrt{Q}\nabla u\cdot\sqrt{Q}\nabla\psi_\alpha dx=\frac{p}{\xi^{p-1}}\int_\Omega|\xi(w+1)\sqrt{Q}\nabla u|^p\chi_\alpha dx,
\]
and the integral on the right-hand side is exactly \(\|\sqrt{Q}\nabla w_\alpha\|_{L^p(\Omega,dx)}^p\).
Combining the identity above with the weak formulation of \eqref{notau} then, we get
\[
    \|\sqrt{Q}\nabla w_\alpha\|_{L^p(\Omega,dx)}^p=\frac{\xi^{p-1}}{p}\int_\Omega f\psi_\alpha vdx.
\]
Inequality \eqref{first} follows from the pointwise estimate \(\psi_\alpha\leq pw_\alpha(w+1)^{p-1}\) for \(p\geq 1\).
\end{proof}

\begin{thm}
Let \(f\in L^{\sigma'}(\Omega,v)\), and let \(u\in QH_0^{1,p}(\Omega,v)\) be a bounded weak solution to \eqref{notau}. Then for \(\xi\) sufficiently small we have
\begin{equation}\label{xpbound}
    \|e^{\xi u}\chi_\alpha\|_{L^{p\sigma}(\Omega,v)}\leq \frac{(1+\alpha)v(\{e^{\xi u}\geq 1+\alpha\})^\frac{1}{p\sigma}}{1-C\xi\|f\|_{L^{\sigma'}(\Omega,v)}^\frac{1}{p-1}}.
\end{equation}
\end{thm}

\begin{proof}
Fixing \(\alpha\geq 0\) and applying Lemma \ref{aux} together with H\"older's inequality gives
\[
    \|\sqrt{Q}\nabla w_\alpha\|^p_{L^p(\Omega,dx)}\leq \xi^{p-1}\| w_\alpha\|_{L^{p\sigma}(\Omega,v)}\|f\|_{L^{\sigma'}(\Omega,v)}\bigg(\int_\Omega(w+1)^{p\sigma}\chi_\alpha vdx\bigg)^\frac{p-1}{p\sigma},
\]
and from \eqref{sobolev} it follows that \(\|w_\alpha\|_{L^{p\sigma}(\Omega,v)}\leq C\|\sqrt{Q}\nabla w_\alpha\|_{L^{p}(\Omega,dx)}\). Combining these inequalities and rearranging, we see that
\[
   \|w_\alpha\|_{L^{p\sigma}(\Omega,v)}\leq  \xi C\|f\chi_\alpha\|_{L^{\sigma'}(\Omega,v)}^\frac{1}{p-1}\|e^{\xi u}\chi_\alpha\|_{L^{p\sigma}(\Omega,v)}.    
\]
Additionally, since \(\|e^{\xi u}\chi_\alpha\|_{L^{p\sigma}(\Omega,v)}=\|w_\alpha+(1+\alpha)\chi_\alpha\|_{L^{p\sigma}(\Omega,v)}\), Minkowski's inequality and the preceding estimate together give
\[
    \|e^{\xi u}\chi_\alpha\|_{L^{p\sigma}(\Omega,v)}\leq \xi C\|f\|_{L^{\sigma'}(\Omega,v)}^\frac{1}{p-1}\|e^{\xi u}\chi_\alpha\|_{L^{p\sigma}(\Omega,v)}+(1+\alpha)v(\{w\geq\alpha\})^\frac{1}{p\sigma}.
\]
For \(\xi\) small (namely, \(\xi <1/C\|f\|_{L^{\sigma'}(\Omega,v)}^\frac{1}{p-1}\)), we can rearrange to obtain \eqref{xpbound}.
\end{proof}

Since the condstants in the above are independent of $\alpha>0$, we may allow $\alpha\rightarrow 0^+$ to obtain Theorem \ref{expint} as a corollary.

%\begin{cor} Let \(u\) be as above. For \(\gamma\) small enough there exists a constant \(C\), independent of \(u\), for which 
%\[
 %   \int_\Omega e^{\gamma u}v dx\leq Cv(\Omega).
%\]
%\end{cor}

%Equipped with a bound of this sort, the authors of \cite{DCU-SR21} produce a refined bound for \(\|u\|_{L^\infty(\Omega,v)}\). An identical argument to theirs could be employed to improve the bound in Theorem \ref{main}, however this would require us to repeat a virtually identical argument to the proof of Theorem \ref{main} for only a marginally better result. Therefore we omit such refinements from this work.

\bibliographystyle{plain}
\bibliography{references}
\end{document}